\documentclass[12pt]{article}
\usepackage{amsmath,amssymb,bm,amsthm,array,tikz,enumerate}
\usetikzlibrary{positioning,graphs}
\DeclareMathOperator{\Gal}{Gal}
\DeclareMathOperator{\Spec}{Spec}
\DeclareMathOperator{\tr}{tr}

\DeclareMathOperator{\Ker}{Ker}
\DeclareMathOperator{\re}{Re}

\newtheorem{thm}{Theorem}[section]
\newtheorem{lem}[thm]{Lemma}

\newtheorem{cor}[thm]{Corollary}
\newtheorem{question}[thm]{Question}

\theoremstyle{definition}
\newtheorem{defi}[thm]{Definition}

\allowdisplaybreaks

\title
{
Regular graphs to induce even periodic Grover walks
}

\author{Sho Kubota\thanks{
Department of Information Systems,
Faculty of Information Science and Technology,
Osaka Institute of Technology,
Osaka 573-0196, Japan.
\texttt{sho.kubota@oit.ac.jp}}
\and Hiroto Sekido\thanks{
Graduate School of Environment Information Sciences, Yokohama National University, Hodogaya,
Yokohama 240-8501 Japan.
\texttt{sekido-hiroto-zk@ynu.jp}}
\and  Kiyoto Yoshino\thanks{
Department of Computer Science,
Faculty of Applied Information Science,
Hiroshima Institute of Technology,
Saeki Ward, Hiroshima, 731-5143, Japan.
\texttt{k.yoshino.n9@cc.it-hiroshima.ac.jp}
}
}
\date{}

\begin{document}
\maketitle
\begin{abstract}
The interest of this paper is a characterization of graphs that induce periodic Grover walks with given periods. In previous studies, Yoshie has shown that the only graphs that induce odd periodic Grover walks are cycle graphs. However, this problem is largely unsolved for even periods. In this study, we show that regular graphs that induce $2l$-periodic Grover walks are also cycle graphs in most cases, where $l$ is an odd integer. The proof uses Galois theory.
\vspace{8pt} 
\end{abstract}

\section{Introduction}

Quantum walks are the quantum mechanical counterpart of classical random walks.
Quantum walks are a type of quantum algorithm that has received significant attention in recent years due to their potential for use in a variety of applications,
such as optimization~\cite{Ya}, search~\cite{Am}, and simulation~\cite{Mo}.

One of the interesting properties of quantum walks is periodicity.
It refers to a phenomenon that quantum states return to their initial states at a certain number of steps, which can occur due to the circumstance that the time evolution matrix of a quantum walk is unitary.
Recently, Panda and Benjamin proposed a quantum cryptographic protocol focusing on periodicity of quantum walks~\cite{AC}, and applications of periodicity have also attracted attention.

The model treated in this paper is the Grover walk, which is a typical discrete-time quantum walk on a graph.
A notable feature of Grover walks is that it is defined from a graph alone.
Not all graphs induce a periodic Grover walk, but rather such graphs are rare.
Hence, we are interested in characterizing graphs that induce periodic Grover walks.
The typical study of periodicity is to characterize graphs
in a specific graph class~\cite{YNIE, S, KSTY, Y2019}.
On the other hand,
there is also the point of view of characterizing graphs
that induce a given period,
and such studies have been carried out by Yoshie~\cite{Y1, Y2}.
An interesting result he showed is that the graphs to induce odd periods are cycle graphs only~\cite{Y2}.
We attempt to develop the work he has done and focus on even periods.
The even-periodic graphs have not been characterized except for $2$-periodic and $4$-periodic graphs.
Graphs that induce $6$ and $8$ periods distinct from cycle graphs are constructed in~\cite{Y1}.
In this study, we restrict graphs we treat to regular graphs and show that regular graphs that induce $2l$-periodic Grover walks are also only cycle graphs in many cases, where $l$ is an odd integer.
More precisely, we show the following.
See later sections for definitions and terminologies.

\begin{thm} \label{0616-1}
Let $G$ be a $k$-regular graph, and let $l$ be an odd integer. 
If $G$ is a $2l$-periodic graph, then one of the following holds.
\begin{enumerate}[(1)]
    \item $G$ is the cycle graph $C_{2l}$ with $2l$ vertices.
    \item $k=3$ and $l$ is a multiple of $3$.
\end{enumerate}
\end{thm}
We remark that this theorem implies if $l$ is not a multiple of $3$,
then the $2l$-periodic regular graph is the cycle graph $C_{2l}$ only.





This paper is organized as follows.
Section~2 is mainly preparation.
Terms and facts on spectral graph theory and Grover walks are introduced.
In Section~3,
we determine the $6$-periodic regular graphs by combining some results of previous studies.
In Section~4, we derive our main theorem (Theorem~\ref{0616-1}).
We study $2l$-periodic regular graphs, where $l$ is an odd integer, by using the Galois group of cyclotomic fields.
Section~\ref{0630-1} summarizes this paper and introduces several open problems.

\section{Preliminaries}

\subsection{Graphs}

Let $G=(V, E)$ be a finite, simple, connected, and undirected graph with vertex set $V$ and edge set $E$. 
For $x, y \in V$, let $xy$ denote the \emph{directed} edge $(x,y) \in V \times V$ from $x$ to $y$. 
Also define $\mathcal{A}=\mathcal{A}(G)$ as the set $\{ xy, yx \mid \{x,y\} \in E\}$. 
Let $a=xy \in \mathcal{A}$ be a directed edge.
Let $a^{-1}$ denote the directed edge $yx$.
Let $o(a)$ and  $t(a)$ be the \emph{origin} $x$ and \emph{terminus} $y$ of $a$, respectively. The \emph{degree} of a vertex $x \in V$ for a graph $G$ is written as $\deg_G x$. The \emph{adjacency matrix} $A=A(G)\in \mathbb{C}^{V\times V}$ of a graph $G$ is defined by
\[A_{x, y}=\begin{cases}
1 & \{x, y\}\in E, \\
0 &\text{otherwise}. 
\end{cases}\]
We define the \emph{degree matrix} $D=D(G) \in \mathbb{C}^{V\times V}$ as $D_{x, y}=(\deg_Gx)\delta_{x, y}$, where $\delta_{x, y}$ is Kronecker's delta. Define the matrix $T=T(G) \in \mathbb{C}^{V\times V}$ by
\[
T=D^{-\frac{1}{2}}AD^{-\frac{1}{2}}.
\]
In particular, $T=\frac{1}{k}A$ holds if the graph $G$ is $k$-regular.
Note that the diagonal entries of $T$ are always zero,
so the sum of the eigenvalues is zero, i.e.,
\begin{equation} \label{0707-2}
\tr T = 0.
\end{equation}
We denote by $\Spec(P)$ the multiset of eigenvalues of a matrix $P$, specifying eigenvalues on the first line and respective multiplicities on the second line as follows.
\begin{align*}
\Spec(P)= 
\begin{bmatrix}
\lambda_0 & \lambda_1 & \cdots & \lambda_{l-1} & \lambda_l \\
a_0 & a_1 & \cdots & a_{l-1} & a_l
\end{bmatrix}.
\end{align*}

\subsection{Grover walks}
In this section, we provide the definitions of the time evolution matrix of a Grover walk and its periodicity, and introduce the spectral mapping theorem. 
Let $G$ be a graph, and set $\mathcal{A}=\mathcal{A}(G)$.
Then the {\it time evolution matrix} $U=U(G) \in \mathbb{C}^{\mathcal{A} \times \mathcal{A} }$ of the Grover walk over $G$ is defined by
\[U_{a, b}=
\begin{cases}
\frac{2}{\deg_Gt(b)}-1  &  a = b^{-1 },\\
\frac{2}{\deg_Gt(b)}  &  t(b)=o(a) \text{ and } a \neq b^{-1}, \\
0 & t(b) \neq o(a).
\end{cases}
\]
Let $\varphi_0 \in \mathbb{C}^\mathcal{A}$ be an initial state of the Grover walk.
Then, the state $\varphi_t$ at time $t$ is expressed as $\varphi_t= U^t\varphi_0$. 
Let $I_\mathcal{A}\in \mathbb{C}^\mathcal{A}$ be the identity matrix. 
If there exists $\tau \in \mathbb{N}$ such that $U^\tau=I_\mathcal{A}$, then we say that the graph $G$ is \emph{periodic} and the minimum $\tau$ is \emph{period}. 
Such a graph is also called a $\tau$-periodic graph.  

\begin{thm}[{\cite[Lemma~5.3]{SHH}}] \label{thm:SMT}
Let $G$ be a graph, and let $U$ be the time evolution matrix of $G$. 
Then, $G$ is periodic if and only if there exists $\tau \in \mathbb{N}$ such that $\lambda^\tau=1$ for any $\lambda \in \Spec(U)$.
\end{thm}

Next, we state a theorem known as the {\it spectral mapping theorem} of Grover walks.

\begin{thm}[\cite{SMT}] 
Let  $G = (V, E)$ be a connected graph.
The multiset of eigenvalues of the time evolution matrix $U
=U(G)$ is expressed as follows.

\[
\Spec(U)= \{e^{\pm i \arccos(\lambda_T)}  | \lambda_T \in \Spec(T) \}\cup [1]^{M_1} \cup [-1]^{M_{-1}}, 
\]
where $M_{1} = |E| - |V| + 1$ and $M_{-1} = |E| - |V| + \dim \Ker (T + I)$.
\end{thm}

\begin{lem} \label{lem:spec} 
    Let $G$ be a connected $k$-regular graph, and set $T:=T(G)$.
    If $G$ is $2l$-periodic, then
    \begin{align*}
        \Spec(T) = 
        \begin{bmatrix}
            1 & \cos(\frac{2\pi}{2l}) & \cos(2 \cdot \frac{2\pi}{2l}) & \cdots & \cos((l-1) \cdot  \frac{2\pi}{2l}) & -1 \\
            1 & a_1 &a_2 & \cdots & a_{l-1} & a_l           
        \end{bmatrix}
    \end{align*}
    for some non-negative integers $a_1,\ldots,a_{l-1}$ and $a_l \in \{0,1\}$.
    Furthermore, $G$ is bipartite if and only if $a_l=1$.
\end{lem}
\begin{proof}
    Assume that the graph $G$ is $2l$-periodic.
    By Theorem~\ref{thm:SMT}, each eigenvalue $\lambda_T \in \Spec(T)$ satisfies $(e^{\pm i \arccos(\lambda_T)})^{2l} = 1$.
    This implies $$\lambda_T = \cos\left( \frac{2j \pi}{2l} \right)$$
    for some integer $j \in \{0,\ldots,l\}$. 
    Also, by the Perron-Frobenius theorem~\cite[Theorem~2.2.1]{SG}, 
    we see that the largest eigenvalue has multiplicity $1$, and derive $a_l \in \{0,1\}$ and the necessary and sufficient condition for $G$ to be bipartite.
\end{proof}

\begin{defi} \label{ss}
Let $A$ be a finite subset of the complex number field $\mathbb{C}$. Then, we define
\begin{align*}
S^n(A)=\sum_{a \in A}a^n 
\end{align*}
for a positive integer $n$. 
For short, we set $S(A)=S^1(A)$.
\end{defi}



Using the handshaking lemma, we have the following.
\begin{lem} \label{handshaking}
Let $G$ be a $k$-regular graph with n vertices, and let $T=T(G)$. 
Then we have
\begin{align}   \label{handshaking:1}
\tr T^2 = \frac{n}{k} . 
\end{align}
\end{lem}

As a basic result,
periodicity of cycle graphs is known as follows.

\begin{lem} \label{0707-1}
The cycle graph $C_n$ with $n$ vertices is $n$-periodic.
\end{lem}
This lemma is essentially derived from~\cite[Theorem~7.1]{SHH} by setting $j=0$ and $p=0$.
Of course, it can also be proved by computing the eigenvalues of the cycle graphs.

\section{$6$-periodic regular graphs} \label{s3}

In this section, we show that the only $6$-periodic regular graph coincides with the cycle graph $C_6$. 
This corresponds to an exceptional case in the main theorem.
As will be discussed in detail in Section~\ref{0706-1},
if $l$ is an odd integer that is a multiple of $3$,
then $2l$-periodic regular graphs could potentially be $3$-regular graphs except for the cycle graph $C_{2l}$.
However, in the case of $l=3$, the $2l$-periodic regular graph is the cycle graph $C_6$ only.

\begin{thm}
    If a connected regular graph $G$ is $6$-periodic, then $G$ is the cycle graph $C_6$.
\end{thm}
\begin{proof}

Let $G$ be a $6$-periodic $k$-regular graph, and set $T=T(G)$.
From Lemma~\ref{lem:spec}, the spectrum of $T$ is written as
\begin{align*}
\Spec(T)= 
\begin{bmatrix}
1 & \cos\frac{\pi}{3} & \cos(\frac{2\pi}{3}) &  -1 \\
1 & a_1 &a_2 &  a_3
\end{bmatrix} 
=
\begin{bmatrix}
1 & \frac{1}{2} & -\frac{1}{2} &  -1 \\
1 & a_1 &a_2 &  a_3
\end{bmatrix} 
\end{align*}
for some non-negative integers $a_1$, $a_2$, and $a_3$.
From the Perron-Frobenius theorem~\cite[Theorem~2.2.1]{SG}, the graph is bipartite if and only if $a_3=1$.
Assume $a_3=1$.
Then, $G$ is a bipartite regular graph,
so we have $a_1 = a_2$.
If $a_1 = a_2 = 0$,
then the graph $G$ is the path graph $P_2$ with $2$ vertices,
but it has period $2$.
If $a_1 = a_2 \geq 1$,
then the bipartite regular graph $G$ has four distinct eigenvalues.
Thus, the result in \cite[Theorem~4.1]{S} yields that $G$ is the cycle graph $C_6$ with $6$ vertices.

Next, assume $a_3=0$.
We have $a_2 \geq 1$ from~\eqref{0707-2}.
If $a_1 = 0$, then~\eqref{0707-2} leads to $a_2 = 2$.
A graph with such spectrum is the complete graph $K_3$ with $3$ vertices only.
However, the period of $K_3$ is $3$~\cite{YNIE}.
Thus, we may assume that $a_1 \geq 1$.
Then
the adjacency matrix of the regular graph $G$ has three distinct eigenvalues,
so $G$ is a strongly regular graph~\cite[Theorem~9.1.2]{SG}.
Periodicity of strongly regular graphs has been investigated in~\cite{YNIE},
and periodic strongly regular graphs can only have $(2k, k, 0, k)$, $(3\lambda, 2\lambda, \lambda, 2\lambda)$, and $(5, 2, 0, 1)$ as parameters.
These are $K_{k,k}, K_{\lambda, \lambda, \lambda}$, and $C_5$ respectively, and none of them has period $6$.
Therefore, if a connected regular graph $G$ is a $6$-periodic graph, then $G$ is the cycle graph $C_6$.
\end{proof}

\section{Regular graphs inducing period $2l$ where $l$ is an odd integer} \label{0706-1}

As the discussion in Section~\ref{s3} suggests,
it is explicitly determined what eigenvalues the matrix $T$ should have for a given period.
In studying larger even periods,
we will obtain information on the eigenvalues using Galois theory.
We assume readers unfamiliar with Galois theory and attempt a brief explanation of the theory.

A complex number $\alpha$ is \emph{algebraic} over $\mathbb{Q}$ if there exists a non-zero polynomial $p(x)$ with coefficients in $\mathbb{Q}$ such that $p(\alpha) = 0$.
Among such polynomials,
the monic one with the smallest degree is called the {\it $\mathbb{Q}$-minimal polynomial} of $\alpha$.
For example,
we see that $\sqrt{2}$ is a root of the polynomial $f(x) = x^2 - 2$, so it is algebraic over $\mathbb{Q}$.
In addition, the polynomial $f(x)$ is the $\mathbb{Q}$-minimal polynomial of $\sqrt{2}$.
If complex numbers $\alpha$ and $\beta$ share the same $\mathbb{Q}$-minimum polynomial,
they are said to be {\it conjugate} over $\mathbb{Q}$,
or we say that $\beta$ is a {\it conjugate} of $\alpha$ over $\mathbb{Q}$.
For example, $\sqrt{2}$ and $-\sqrt{2}$ are conjugate over $\mathbb{Q}$.
In general, if a complex number $\alpha$ is a root of a $\mathbb{Q}$-coefficient polynomial $p(x)$,
then all conjugates $\alpha_1, \dots, \alpha_s$ over $\mathbb{Q}$ of $\alpha$ are roots of $p(x)$.
This implies that the multiplicities of $\alpha_1, \dots, \alpha_s$ are equal for the roots of the $\mathbb{Q}$-coefficient polynomial $p(x)$.
Therefore, capturing conjugates over $\mathbb{Q}$ of algebraic numbers is useful for partially catching roots of polynomials.

To capture information on conjugates,
it is useful to study symmetries (automorphisms) of fields.
We take $\sqrt{2}$ and its conjugate over $\mathbb{Q}$,
$-\sqrt{2}$, as an example.
The automorphism $\mathbb{Q}(\sqrt{2}) \to \mathbb{Q}(\sqrt{2})$ that does not change elements of $\mathbb{Q}$ is only $\sigma : \sqrt{2} \mapsto -\sqrt{2}$ except for the identity mapping.
In fact, this mapping $\sigma$ provides us with information on the conjugates over $\mathbb{Q}$ of $\sqrt{2}$.
In the following paragraph,
we supplement terminologies on field theory.
Essential importance in this study is Lemma~\ref{kyoyaku}.

Let $L$ and $K$ be fields.
We denote $L \supset K$ by $L/K$ and say that it is a \emph{field extension}.
Let $\zeta_n = e^\frac{2\pi i}{n}$ for each positive integer $n$.
It is known that the field extension $\mathbb{Q}(\zeta_n) / \mathbb{Q}$ is a {\it Galois extension}.
This paper does not define Galois extensions,
but we do not need to understand them in depth either.
The set of all automorphisms of $\mathbb{Q}(\zeta_n)$ that do not change any elements of $\mathbb{Q}$ is a group with respect to the composition of maps.
We call the set the \emph{Galois group} of $\mathbb{Q}(\zeta_n) / \mathbb{Q}$ and denote it by $\Gal(\mathbb{Q}(\zeta_n)/\mathbb{Q})$.



\begin{lem} [{\cite[Corollary~5.1.8]{Gal}}]  \label{kyoyaku}
Let $\alpha \in \mathbb{Q}(\zeta_n)$. Then, the following are equivalent.
\begin{enumerate}[(1)]
    \item $\beta \in \mathbb{Q}(\zeta_n)$ is a conjugate of $\alpha$ over $\mathbb{Q}$. \label{kyoyaku:1}
    \item There exists $f \in \Gal(\mathbb{Q}(\zeta_n)/\mathbb{Q})$ such that $f(\alpha)=\beta$. \label{kyoyaku:2}
\end{enumerate}
\end{lem}

Next, we explain the Galois group $\Gal(\mathbb{Q}(\zeta_n)/\mathbb{Q})$. 
Let $n \geq 2$ be a positive integer.
Let $(\mathbb{Z}/n\mathbb{Z})^{\times}$ be the set of units of $\mathbb{Z}/n\mathbb{Z}$.
For each $m \in (\mathbb{Z}/n\mathbb{Z})^{\times}$, we write $\sigma_m$ for the automorphism of $\mathbb{Q}(\zeta_n)$ such that
\begin{align*}
    \sigma_m(\zeta_n) = \zeta_n^m.
\end{align*}
We note that elements of Galois groups are ring homomorphisms that do not change the elements of $\mathbb{Q}$,
so they are defined by specifying only the image of the generators.
It is known that
\[
\Gal(\mathbb{Q}(\zeta_n)/\mathbb{Q})= \{ \sigma_m \mid m \in (\mathbb{Z}/n\mathbb{Z})^{\times} \}.
\]

We determine the $2l$-periodic regular graphs where $l$ is an odd integer.
For an integer $j$, we let $\alpha^{(n)}_j := \cos\big(\frac{2\pi}{n}j \big) = \frac{1}{2}(\zeta^j_n + \zeta^{-j}_n)$.



We define $G_n:=\Gal(\mathbb{Q}(\zeta_{n})/\mathbb{Q} )$. 
For $x \in \mathbb{Q}(\zeta_{n})$,
consider the set $G_nx=\{\sigma(x) \mid \sigma \in \Gal(\mathbb{Q}(\zeta_{n})/\mathbb{Q} )\}$. From Lemma~\ref{kyoyaku}, the elements of $G_n\alpha^{(n)}_j$ are conjugates of $\alpha^{(n)}_j$ over $\mathbb{Q}$. 

The following series of lemmas is necessary computations for proving our main theorem.
The essential reason why our main theorem is derived can be explained as that Lemma~\ref{handshaking} imposes a strong restriction to $2l$-periodic regular graphs,
where $l$ is an odd integer.
Hence, we are interested in the sum of squares of eigenvalues.
Since eigenvalues that are conjugate over $\mathbb{Q}$ to each other have the same multiplicity,
the value of $S^2(G_n \alpha_{j}^{(n)})$ for each $j$ is key.
To compute this quantity,
we first consider properties of orbits containing $\zeta_n^j$.

\begin{lem} \label{w}
Let $n$ be a positive integer,
and let $m$ be a positive divisor of $n$. Then we have
\[
    G_n\zeta_n^m= G_{\frac{n}{m}} \zeta_{\frac{n}{m}}
    \quad \text{ and } \quad
    G_n\alpha^{(n)}_m= G_{\frac{n}{m}} \alpha^{\left(\frac{n}{m}\right)}_1.
\]
In particular,
$
S(G_n\zeta_n^m)= S(G_{\frac{n}{m}} \zeta_{\frac{n}{m}}).
$
\end{lem}
\begin{proof}
By Lemma~\ref{kyoyaku}, we have
\begin{align*}
    G_n\zeta_n^m
    & = \{ z \in \mathbb{C} \mid \text{$z$ and $\zeta_n^m$ are conjugate over $\mathbb{Q}$}  \} \\
    &= \{ z \in \mathbb{C} \mid \text{$z$ and $\zeta_{\frac{n}{m}}$ are conjugate over $\mathbb{Q}$}  \} \\
    &= G_{\frac{n}{m}}\zeta_{\frac{n}{m}}.
\end{align*}
Also, $G_n\alpha^{(n)}_m= G_{\frac{n}{m}} \alpha^{\left(\frac{n}{m}\right)}_1$ holds in the same manor.
\end{proof}

The function $\phi : \mathbb{Z}_{> 0} \to \mathbb{Z}_{> 0}$ defined as $\phi(n)=|(\mathbb{Z}/n\mathbb{Z})^{\times}|$ is the {\it Euler function}.
Note that 
\begin{equation} \label{0716-1}
\phi(2n) = \phi(n)
\end{equation}
if $n$ is odd.

\begin{lem} \label{thnumber}
Let $n \geq 3$ be an integer.
Then, the mapping $\re: G_n \zeta_n \to G_n \alpha_1^{(n)}$,
which is defined by $\re(\zeta_n^j) = \frac{\zeta_n^j + \zeta_n^{-j}}{2} = \alpha_j^{(n)}$ for $j \in (\mathbb{Z}/n\mathbb{Z})^{\times}$,
is a $2$-to-$1$ mapping,
i.e., for any $\alpha_j^{(n)} \in G_n \alpha_1^{(n)}$,
we have $|\re^{-1}( \{ \alpha_j^{(n)} \})| = 2$.
Also, let $m \geq 3$ be a divisor of $n$. Then, we have

\[
|G_n\alpha_1^{(m)}|=\frac{\phi(m)}{2} .
\]
\end{lem}

\begin{proof}

From the definition of $\alpha_j^{(n)}$, the mapping is a surjection. 
Next, we fix $\alpha_k^{(n)} \in G_n \alpha_1^{(n)}$, where $k \in (\mathbb{Z}/n\mathbb{Z})^{\times}$.
We assume that $\re(\zeta_n^j) = \alpha_k^{(n)}$, where $j \in (\mathbb{Z}/n\mathbb{Z})^{\times}$.
Then, 
\begin{align*}
0 &= \re(\zeta_n^j) - \alpha_k^{(n)}
= \alpha_j^{(n)}-\alpha_k^{(n)} =\cos\big(\frac{2\pi}{n}j\big) - \cos\big(\frac{2\pi}{n}k\big) \\
&=-2\sin \big(\frac{j+k}{n}\pi \big)\sin \big(\frac{j-k}{n}\pi \big).
\end{align*}
Therefore, we have $j = \pm k$.
Moreover, we have $k \neq -k$.
Indeed, if $k = -k$, then $2k = 0$.
Since $k \in (\mathbb{Z}/n\mathbb{Z})^{\times}$,
we have $2=0$.
This contradicts $n \geq 3$.
Thus $\zeta_n^k \neq \zeta_n^{-k}$. 
We see that the mapping $\re$ is a $2$-to-$1$ mapping. 

Also, let $l$ be a positive integer satisfying $n=ml$.
By Lemma~\ref{w}, we have
\[
G_n\zeta_m = G_n \zeta_{n}^l = G_{\frac{n}{l}} \zeta_{\frac{n}{l}} = G_m\zeta_m. 
\]
Hence,
\[
G_n\alpha_1^{(m)}= \re(G_n\zeta_m)=\re(G_m\zeta_m).
\]
Since the mapping $\re : G_m\zeta_m \to G_m\alpha_1^{(m)}$ is a $2$-to-$1$ mapping as $m \geq 3$, 
we have 
\[
|G_n\alpha_1^{(m)}|
= |\re(G_m \zeta_m) |
=\frac{|G_m \zeta_m|}{2}
=\frac{\phi(m)}{2}.
\]
\end{proof}

\begin{defi}
For a positive integer $n$, let
\[
D(n) = \{m \in \mathbb{Z} \mid \text{$m$ is a positive divisor of $n$} \} .
\]
\end{defi}

\begin{lem} \label{disjoint}
    Let $n$ be a positive integer.
    For distinct integers $i \neq j$ in $D(n)$, $G_n \zeta_n^i \cap G_n \zeta_n^j = \emptyset$.
\end{lem}
\begin{proof}
    Fix such integers $i$ and $j$.
    By way of contradiction, we assume $G_n \zeta_n^i \cap G_n \zeta_n^j \neq \emptyset$.
    Then there exist $x$ and $y$ in $\mathbb{Z}$ coprime to $n$ such that $ix-jy \equiv 0 \pmod n$.
    Hence $ix \equiv 0 \pmod j$.
    Without loss of generality, we may assume that $j \nmid i$.
    Then we see that $\gcd(n,x) \geq \gcd(j,x) \geq 2$.
    This contradicts the assumption that $x$ and $n$ are coprime.
\end{proof}

Let $n$ be a positive integer. The function $\mu (n)$ that satisfies the following is called the {\it M\"{o}bius function}.
\[
\mu(n)= \begin{cases}
(-1)^r & \text{if $n$ is a product of $r$ different prime numbers, } \\
0 & \text{if $n$ has at least one square factor. }
\end{cases}
\]

\begin{lem}[{\cite[(16.6.4)]{ITN}}] \label{wa}
Let $n$ be a positive integer. Then, we have 
\[
S(G_n\zeta_{n}) = \mu(n). 
\]
\end{lem}

\begin{lem} \label{zetasquare}
Let $n$ be a positive integer.
Then, we have

\[
S( G_n\zeta_n^2) = \begin{cases}
S( G_n\zeta_n) & \text{if $n$ is odd,}\\
S( G_\frac{n}{2}\zeta_\frac{n}{2}) & \text{if  $n$ is even.}
\end{cases}
\]
\end{lem}
\begin{proof}

If $n$ is odd, then $n$ and $2$ are coprime. 
We have $2 \in (\mathbb{Z}/n\mathbb{Z})^{\times}$ and $\sigma_2 \in G_n$.
Therefore $\zeta_n$ and $\zeta_n^2$ are conjugate over $\mathbb{Q}$ from Lemma~\ref{kyoyaku}. Thus
\[
G_n\zeta_n^2 =G_n\zeta_n .
\]
When $n$ is even,
we have
\[
G_n\zeta_n^2 = G_{\frac{n}{2}}\zeta_{\frac{n}{2}}
\]
by Lemma~\ref{w}.
\end{proof}

\begin{lem}\label{aa}
Let $m$ be an integer, and let $n$ be a positive integer, we have 
\begin{align}
S^2(G_n\zeta_n^m) = S^2(G_n\zeta_n^{-m}). \label{aaaa}
\end{align} 
\end{lem}

\begin{proof}

Since $n$ and $n-1$ are coprime, $\zeta_n$ and $\zeta_n^{n-1}=\zeta_n^{-1}$ are conjugate over $\mathbb{Q}$. 
Hence $G_n \zeta_n^m = G_n (\zeta_n^{-1})^m = G_n \zeta_n^{-m}$ by Lemma~\ref{kyoyaku}.
This implies $S^2(G_n\zeta_n^m) = S^2(G_n\zeta_n^{-m})$.
\end{proof}

\begin{cor} \label{sps}
Let $l \geq 3$ be an odd integer, and set $n := 2l$. Then we have
\[
S^2(G_n\alpha_1^{(n)})= \frac{1}{4}\{ S(G_n\zeta_n^2) + \phi(n) \} . 
\]
\end{cor}

\begin{proof}

From Lemma \ref{thnumber}, we have 
\begin{align*}
S^2(G_n\alpha_1^{(n)}) 
= \sum_{\alpha \in G_n \alpha_1^{(n)}} \alpha^2 
= \frac{1}{2} \sum_{z \in G_n \zeta_n} \left(\frac{z+z^{-1}}{2}\right)^2.
\end{align*}
We see that
\begin{align*}
\frac{1}{2} \sum_{z \in G_n \zeta_n} \left(\frac{z+z^{-1}}{2}\right)^2
&=
\frac{1}{2} \sum_{j \in (\mathbb{Z}/n\mathbb{Z})^\times} \left(\frac{\zeta_n^j+\zeta_n^{-j}}{2}\right)^2 \\
&=\frac{1}{8}\sum_{j \in (\mathbb{Z}/n\mathbb{Z})^\times}\{(\zeta_n^{j})^2 + (\zeta_n^{-j})^2+2\}\\
&=\frac{1}{8}\sum_{j \in (\mathbb{Z}/n\mathbb{Z})^\times}\{(\zeta_n^{2})^j + (\zeta_n^{-2})^j+2\}\\
&=\frac{1}{8}\left\{S(G_n\zeta_n^2) + S(G_n\zeta_n^{-2}) + 2\phi(n)\right\}  \\
&=\frac{1}{8}\{2S(G_n\zeta_n^2)  + 2\phi(n)\}  \\
&=\frac{1}{4}\{S(G_n\zeta_n^2)  + \phi(n)\}, \notag
\end{align*}
where the fifth equality is derived from Lemma~\ref{aa}. 
\end{proof}

\begin{thm}[Theorem~\ref{0616-1}, restated] \label{0629-1}
Let $G$ be a $k$-regular graph, and let $l$ be an odd integer. 
If $G$ is a $2l$-periodic graph, then one of the following holds.
\begin{enumerate}[(1)]
    \item $G$ is the cycle graph $C_{2l}$ with $2l$ vertices.
    \item $k=3$ and $l$ is a multiple of  $3$.
\end{enumerate}

\end{thm}

\begin{proof}

Let $l=p_1^{r_1}p_2^{r_2}\cdots p_s^{r_s}$ be the prime factorization of $l$, where $p_i < p_{i+1}$.
By Lemma~\ref{lem:spec}, the eigenvalues of $T$ are elements of $\{ \alpha^{(2l)}_j \mid j \in \{0,\ldots,l \}\}$.
Noting that the characteristic polynomial of $T$ is a $\mathbb{Q}$-coefficient polynomial, we see that a conjugate of an eigenvalue of $T$ is an eigenvalue.
Hence, the set of eigenvalues of $T$ coincides with the union of a subset of $\{ G_n\alpha^{(2l)}_j \mid j \in D(2l) \}$.
Here note that the conjugate classes are disjoint by Lemma~\ref{disjoint}.
Since pairwise conjugate eigenvalues over $\mathbb{Q}$ have the same multiplicity, we write $m_d$ ($d \in D(2l)$) for the multiplicity of an eigenvalue in $G_n\alpha^{(2l)}_d$.
Let $X_{d'}= m_{d'} + m_{2d'}$ for $d' \in D(l)$. 
Since $T$ is an $n\times n$ matrix,
we have
\begin{align*}
n
&= \sum_{d\in D(2l)} |G_{2l}\alpha_d^{(2l)}|m_d  \\
& = \sum_{d\in D(2l)} |G_{\frac{2l}{d}}\alpha_1^{\left(\frac{2l}{d}\right)}| m_d \\
&= \frac{1}{2} \sum_{d\in D(2l)}\phi\left(\frac{2l}{d}\right)m_d \\
&= \frac{1}{2}\left\{\sum_{d'\in D(l)}\phi\left(\frac{2l}{d'}\right)m_{d'} + \sum_{d'\in D(l)}\phi\left(\frac{2l}{2d'}\right)m_{2d'}\right\},
\end{align*}
where the second and third equalities follow from Lemmas~\ref{w} and~\ref{thnumber}, respectively.
Since $l$ is odd,
Equality~\eqref{0716-1} leads to
\begin{align}
\phi\left(\frac{2l}{d'}\right)&=\phi\left(\frac{l}{d'}\right)  \label{g} 
\end{align}
for $d' \in D(l)$.
Thus, 
\begin{align*}
&\frac{1}{2}\left\{\sum_{d'\in D(l)}\phi\left(\frac{2l}{d'}\right)m_{d'} + \sum_{d'\in D(l)}\phi\left(\frac{2l}{2d'}\right)m_{2d'}\right\} \\
&= \frac{1}{2}\sum_{d'\in D(l)}\phi\left(\frac{l}{d'}\right)(m_{d'} + m_{2d'}) \\
&=  \frac{1}{2}\sum_{d'\in D(l)}\phi\left(\frac{l}{d'}\right)X_{d'}.
\end{align*}
Therefore, we have
\begin{align}
    n = \frac{1}{2}\sum_{d'\in D(l)}\phi\left(\frac{l}{d'}\right)X_{d'}. \label{n}
\end{align}
Also, we write $S^2(\Spec(T))$ for the sum $\sum_{\lambda} a_\lambda \lambda^2$ where $\lambda$ is an eigenvalue of $T$ and $a_\lambda$ is the multiplicity of $\lambda$. We have
\begin{align*}
S^2(\Spec(T)) &= \sum_{d \in D(2l)}S(G_{2l}(\alpha_d^{(2l)})^2)m_d \\
&= \sum_{d \in D(2l)}S(G_{\frac{2l}{d}}(\alpha_1^{(\frac{2l}{d})})^2)m_d \\
&= \frac{1}{4} \sum_{d \in D(2l)}    \left\{ S(G_{\frac{2l}{d}}\zeta_{\frac{2l}{d}}^2) + \phi\left(\frac{2l}{d}\right) \right\} m_d, 
\end{align*}
where the third equality is derived from Corollary~\ref{sps}. 
If $d \in D(2l)$ is even,
it can be displayed as $d=2d'$ for some $d' \in D(l)$.
Then we have
\begin{align*}
\left\{ S(G_{\frac{2l}{d}}\zeta_{\frac{2l}{d}}^2) + \phi\left(\frac{2l}{d}\right) \right\} m_d
&= \left\{ S(G_{\frac{2l}{2d'}}\zeta_{\frac{2l}{2d'}}^2) + \phi\left(\frac{2l}{2d'}\right) \right\} m_{2d'} \\
&= \left\{ S(G_{\frac{l}{d'}}\zeta_{\frac{l}{d'}}^2) + \phi\left(\frac{l}{d'}\right) \right\} m_{2d'} \\
&= \left\{ S(G_{\frac{l}{d'}}\zeta_{\frac{l}{d'}}) + \phi\left(\frac{l}{d'}\right) \right\} m_{2d'},
\end{align*}
where the last equality is derived from Lemma~\ref{zetasquare}.
If $d \in D(2l)$ is odd, we have $d \in D(l)$,
so Lemma~\ref{w} and Equality~\eqref{g} yield 
\[ \left\{ S(G_{\frac{2l}{d}}\zeta_{\frac{2l}{d}}^2) + \phi\left(\frac{2l}{d}\right) \right\} m_d
= \left\{ S(G_{\frac{l}{d}}\zeta_{\frac{l}{d}}) + \phi\left(\frac{l}{d}\right) \right\} m_d \]
Thus, by converting the variable $d$ to $d'$, we have
\begin{align*}
S^2(\Spec(T)) 
&= \frac{1}{4} \sum_{d \in D(2l)}    \left\{ S(G_{\frac{2l}{d}}\zeta_{\frac{2l}{d}}^2) + \phi\left(\frac{2l}{d}\right) \right\} m_d \\
&= \frac{1}{4} \sum_{d' \in D(l)}   \left\{ S(G_{\frac{l}{d'}}\zeta_{\frac{l}{d'}}) + \phi\left(\frac{l}{d'}\right) \right\}   (m_{d'} + m_{2d'}) \\
&= \frac{1}{4} \sum_{d' \in D(l)}   \left\{ S(G_{\frac{l}{d'}}\zeta_{\frac{l}{d'}}) + \phi\left(\frac{l}{d'}\right) \right\} X_{d'} \\
&= \frac{1}{4} \sum_{d' \in D(l)}   \left\{ \mu\left(\frac{l}{d'}\right) + \phi\left(\frac{l}{d'}\right) \right\} X_{d'},
\end{align*}
where the last equality is derived from Lemma~\ref{wa}.
Therefore, we have
\begin{align}
    S^2(\Spec(T))
    = \frac{1}{4} \sum_{d' \in D(l)}   \left\{ \mu\left(\frac{l}{d'}\right) + \phi\left(\frac{l}{d'}\right) \right\} X_{d'} . \label{mphi}
\end{align}
Substituting Equalities~\eqref{n} and~\eqref{mphi} into Lemma~\ref{handshaking}~\eqref{handshaking:1}, we obtain 
\[
\sum_{d' \in D(l)}  \left[ k\left\{ \mu\left(\frac{l}{d'}\right) + \phi\left(\frac{l}{d'}\right) \right\} -2\phi\left(\frac{l}{d'}\right) \right] X_{d'} =0.
\]
Note that $\mu(1)=1, \phi(1)=1$.
We have
\begin{align*}
\sum_{d' \in D(l) \setminus \{ l \} }  \left[ k\left\{ \mu\left(\frac{l}{d'}\right) + \phi\left(\frac{l}{d'}\right) \right\} -2\phi\left(\frac{l}{d'}\right) \right] X_{d'} 
&=  -\left[k\{ \mu(1) + \phi(1) \} -2\phi(1) \right] X_{l} \\
&=  -2(k-1)X_{l}.
\end{align*}
We remark that $X_l$ is positive.
Therefore, we see that the graph does not satisfy
\begin{align} \label{contraposition}
k\left\{ \mu\left(\frac{l}{d'}\right) + \phi\left(\frac{l}{d'}\right) \right\} -2\phi\left(\frac{l}{d'}\right) \geq 0
\quad
\text{for any}
\quad
d'\in D(l) \setminus \{ l \}.
\end{align}
We have
\begin{align*}
\eqref{contraposition}
&\iff
k \geq \frac{2\phi(\frac{l}{d'})}{\mu(\frac{l}{d'}) + \phi(\frac{l}{d'})}
\quad
\text{for any}
\quad
d'\in D(l) \setminus \{ l \} \notag  \\ 
&\iff
k \geq \max\left\{ \frac{2\phi(\frac{l}{d'})}{\mu(\frac{l}{d'}) + \phi(\frac{l}{d'})} \, \middle| \, d' \in D(l) \setminus \{l\}\right\}.
\end{align*}
The value
\[ \frac{2\phi(\frac{l}{d'})}{\mu(\frac{l}{d'}) + \phi(\frac{l}{d'})} = 2-\frac{2\mu(\frac{l}{d'})}{\mu(\frac{l}{d'}) + \phi(\frac{l}{d'})} \]  
takes the maximum when $\phi(\frac{l}{d'})$ takes the minimum and $\mu(\frac{l}{d'}) = -1$.
Since this is attained when $\frac{l}{d'}=p_1$,
we have
\begin{align*}
\max\left\{ 2- \frac{ 2\mu(\frac{l}{d'}) }{\mu(\frac{l}{d'}) + \phi(\frac{l}{d'})} \, \middle| \, d' \in D(l) \setminus \{l\}\right\}
=2+ \frac{ 2 }{\phi(p_1) -1 } 
=2+ \frac{ 2 }{p_1-2}.
\end{align*} 
As a result, \eqref{contraposition} is equivalent to
\begin{equation} \label{0623-1}
k \geq 2+ \frac{ 2 }{p_1-2} .
\end{equation}
This means that if $k$ satisfies~\eqref{0623-1},
then the graph is not periodic.
Therefore,
we conclude that $k \leq 3$ if $p_1 = 3$, and $k=2$ if $p_1 \geq 5$.
Therefore, either $k=3$ and $l$ is a multiple of $3$, or $k=2$ holds. In particular, if $k=2$, then $G$ is the cycle graph,
and hence Lemma~\ref{0707-1} implies that
the cycle graph $G$ has $2l$ vertices.
\end{proof}

It should be noted that Theorem~\ref{0629-1} refers only to the necessary condition for $2l$-periodic regular graphs.
We already know that
the cycle graph with $2l$ vertices is $2l$-periodic,
but we know nothing so far on the situation where $k=3$ and $l$ is a multiple of $3$.
The existence of such graphs is our future work:

\begin{question}
Let $l$ be an odd integer that is a multiple of $3$.
Do $2l$-periodic $3$-regular graphs exist?
\end{question}

Note that only the complete bipartite graph $K_{3,3}$
is currently discovered as periodic 3-regular graphs,
and its period is $4$~\cite{YNIE}.

\section{Summary and $4s$-periodic regular graphs} \label{0630-1}
Once again, let $l$ be an odd integer.
In this paper, we proved that only the $2l$-periodic regular graphs are cycle graphs $C_{2l}$ with $2l$ vertices
if $l$ is not a multiple of $3$.
There could potentially be $2l$-periodic $3$-regular graphs
only if $l$ is a multiple of $3$.

Incidentally, we did not make any mention of the case where $l$ is even in this paper.
In the remainder, we will mention periodic regular graphs whose period is a multiple of $4$.
Let $s$ be a positive integer.
In fact, many examples of $4s$-periodic regular graphs
have already been found,
for example, $12$-periodic regular graphs are summarized
in Table~2 of~\cite{S}.
Moreover, if $G$ is a $4s$-periodic graph,
then $G \otimes J_t$, which is the Kronecker product of $G$ and the square all-one matrix $J_t$ of size $t$,
is also a $4s$-periodic graph.
In particular, $C_{4s} \otimes J_t$ is a $4s$-periodic regular graph for any $t$.
For small $s$, we know the following.
The only $4$-periodic graphs are complete bipartite graphs~\cite{Y1}.
For $8$-periodic regular graphs,
it is known that graphs are constructed from transversal designs,
besides $C_{8} \otimes J_t$ (see Table~3 in~\cite{S}).
For $12$-periodic regular graphs,
see Table~2 in~\cite{S} as already mentioned.
For $20$-periodic regular graphs,
some examples have been found besides $C_{20} \otimes J_t$,
and they appear in Section~9 of~\cite{C}.
At least for small $s$,
there are sporadic examples of $4s$-periodic regular graphs
besides $C_{4s} \otimes J_t$.
We expect that characterizing $4s$-periodic regular graphs will be difficult,
but we believe it will also be an interesting problem.



\section*{Acknowledgements}
We would like to thank Professor Norio Konno for his fruitful comments and helpful advice.
S.K. is supported by JSPS KAKENHI (Grant Number JP20J01175).
K.Y is supported by JSPS KAKENHI (Grant Number JP21J14427).

\end{document}